\documentclass[onefignum,onetabnum]{siamart220329}

\usepackage{amsfonts}
\usepackage{graphicx}
\usepackage{subcaption}
\usepackage{epstopdf}
\usepackage{hyperref}
\usepackage{nameref}
\usepackage{algorithmic}
\usepackage{eqparbox}

\ifpdf
  \DeclareGraphicsExtensions{.eps,.pdf,.png,.jpg}
\else
  \DeclareGraphicsExtensions{.eps}
\fi

\newsiamremark{remark}{Remark}
\newsiamremark{hypothesis}{Hypothesis}
\crefname{hypothesis}{Hypothesis}{Hypotheses}
\newsiamthm{claim}{Claim}

\headers{A spectral preconditioner for CG}{Y. Diouane, S. G\"urol, O. Mouhtal and D. Orban}

\def\mytitle{%
  A Spectral Preconditioner for the Conjugate Gradient Method with Iteration Budget
}
\title{\mytitle}

\author{Y. DIOUANE\thanks{GERAD and Department of Mathematics and Industrial Engineering, Polytechnique Montr\'eal. 
(\email{youssef.diouane@polymtl.ca},
\email{dominique.orban@polymtl.ca}).}
\and S.G\"urol\thanks{CERFACS / CECI
CNRS UMR 5318, Toulouse, France.
  (\email{gurol@cerfacs.fr}, \email{mouhtal@cerfacs.fr}).
  This work was funded by French National Programme LEFE/INSU.}
\and  O.MOUHTAL\footnotemark[3] \footnotemark[2]  
\and D.ORBAN\footnotemark[2]  
}

\usepackage{amsopn}
\DeclareMathOperator{\diag}{diag}

\usepackage[normalem]{ulem}
\usepackage[numbers]{natbib}
\usepackage{mathtools}
\usepackage{booktabs}
\usepackage{dsfont}
\usepackage{comment}
\usepackage{enumitem}
\usepackage[
  createShortEnv,
  conf={one big link},
  commandRef=Cref]{proof-at-the-end}

\NewDocumentEnvironment{proofEE}{O{}+b}{%
  \begin{proofE}[#1]
    #2
    \space
  \end{proofE}
}{}

\newcommand{\R}{\mathds{R}}

\newcommand{\om}[1]{\textcolor{red}{#1}}

\newcommand*{\bdg}{\ell_{\max}}

\ifpdf
\hypersetup{
  pdftitle={\mytitle},
  pdfauthor={Youssef Diouane, Selime G\"urol, Oussama Mouhtal and Dominique Orban}
}
\fi

\begin{document}

\maketitle
\pagestyle{myheadings}

\begin{abstract}
We study the solution of large symmetric positive-definite linear systems in a matrix-free setting with a limited iteration budget.
We focus on the preconditioned conjugate gradient (PCG) method with spectral preconditioning. Spectral preconditioners map a subset of eigenvalues to a positive cluster via a scaling parameter, and leave the remainder of the spectrum unchanged, in hopes to reduce the number of iterations to convergence.
We formulate the design of the spectral preconditioners as a constrained optimization problem.
The optimal cluster placement is defined to minimize the error in energy norm at a fixed iteration. This optimality criterion provides new insight into the design of efficient spectral preconditioners when PCG is stopped short of convergence.
We propose practical strategies for selecting the scaling parameter, hence the cluster position, that incur negligible computational cost.
Numerical experiments highlight the importance of cluster placement and demonstrate significant improvements in terms of error in energy norm, particularly during the initial iterations.
\end{abstract}



\section{Introduction}

We consider large symmetric positive-definite (SPD) linear systems, where the matrix is implicitly represented as an \emph{operator}, using the conjugate gradient (CG) method \citep{Hestenes1952MethodsOC}. A standard approach to accelerate convergence is to introduce a preconditioner. There is no universal strategy for constructing effective preconditioners; we refer the reader to the surveys \citep{Preconditionning1,PearsonPestana2020,wathen2015preconditioning} for an overview.

One important class of preconditioners adapted to matrix-free settings is that of limited memory preconditioners (LMPs) \citep{LMP}. This approach is closely related to deflation~\citep{Frank, DeflatedCG}, augmentation methods \citep{coulaud2013deflation, Gaul}, balancing strategies \citep{nabben2006comparison}, and low-rank updates~\citep{Bergamaschi2020}. The main idea is to modify the eigenspectrum of the preconditioned matrix such that a small set of eigenvalues of the linear system matrix are either mapped to a positive scalar value \citep{LMP} or are removed \citep{Frank} in hopes to reduce the number of iterations to converge.

In this work, we focus on \textit{spectral preconditioners} \citep{carpentieri2003class, SenSpectLMP, LMP}, which are constructed using extreme eigenvalues and corresponding eigenvectors. These preconditioners come in two different types. The first type maps the extreme eigenvalues to a positive scalar, and is called a \textit{deflating preconditioner} in \citep{SenSpectLMP} and a \textit{spectral LMP} in \citep{LMP, TshiGratWeavSart08}. The second type shifts the extreme eigenvalues by a common positive value, and is called a coarse grid preconditioner \citep{carpentieri2003class}. In practice, the scalar value is often set to one~\citep{SenSpectLMP,LMP, TshiGratWeavSart08}. 

Both approaches are considered in \citep{SenSpectLMP}, where the authors analyze the eigenvalue distribution of the preconditioned systems to first-order approximation when the eigenspectrum is approximately known.
For instance, in large-scale weather prediction \citep{DataAssimilaton}, the spectral LMP is constructed using Ritz pairs extracted from PCG recurrences \citep{YoussSaad}. Alternatively, parallel randomized algorithms \citep{Halko} can be used to approximate extreme eigenvalues and their corresponding eigenvectors \citep{frangella2023randomized}. 

Our main objective is to design a spectral preconditioner that minimizes the error after $\ell$ steps, where $\ell$ is fixed, motivated by applications in which the solver must be stopped after a prescribed number of iterations. The resulting preconditioner is closely related to the deflating preconditioner \citep{SenSpectLMP, LMP} and is also connected to deflation techniques \citep{DeflatedCG}.

Formulating the requirements of the preconditioner as an optimization problem reveals its structure: the preconditioner maps a selected subset of eigenvalues to positive values while leaving the rest unchanged. Since explicitly computing this set of eigenvalues is impractical, we propose a strategy to identify the relevant part of the spectrum. We show that the preconditioner should be constructed from the extreme eigenvalues, and we propose several strategies to define the positioning of the cluster. These strategies are also shown to be linked to the energy norm of the error, which is precisely the quantity minimized by CG.

The paper is organized as follows. In \Cref{sec:Background}, we review the CG method and its convergence properties, and then discuss the characteristics of an efficient preconditioner. Our main contributions are presented in \Cref{sec:A scaled spectral preconditioner,sec:Choices of the scalling}, where we define the \emph{scaled} spectral preconditioner and analyze its properties. We also outline four strategies for selecting the scaling parameter, which determines the placement of the eigenvalue cluster, to \textit{reduce the total number of iterations and improve convergence in the early stages}. In \Cref{sec:NumExp}, we report numerical experiments on matrices with extreme eigenvalues. Conclusions and perspectives appear in \Cref{sec:conclusion}.

\section*{Notation}
\label{sec:notation}
Matrix $A \in \R^{n \times n}$ is always SPD.
Its spectral decomposition is $A = S \Lambda S^\top$ with $\Lambda = \text{diag}(\lambda_1, \ldots, \lambda_n)$, $ \lambda_1\geq \ldots \geq\lambda_n > 0$, and $S = \begin{bmatrix} s_1 & \cdots & s_n \end{bmatrix}$ being orthogonal. The notation $\mathrm{diag}(\lambda_1, \ldots, \lambda_n)$ denotes the diagonal 
matrix whose diagonal entries are $\lambda_1, \ldots, \lambda_n$. The $A$-norm, or \emph{energy norm}, of $x \in \mathbb{R}^n$ is $\|x\|_{A} =\sqrt{x^\top A x}$. The identity matrix of size $n$ is $I_n $. The set of eigenvalues of $X = X^T \in \mathbb{R}^{n\times n}$ is $\operatorname{spec}(X)$.
The cardinality of a subset $\pi \subseteq \{1, \ldots, n\}$ is $|\pi|$.

\bigskip

In what follows, all results are derived under the assumption of exact arithmetic.
\section{Background}
\label{sec:Background}
\subsection{Conjugate Gradient}
The CG \citep{Hestenes1952MethodsOC} is a method for $Ax = b$ with $A \in \R^{n \times n}$ SPD  and $b \in \R^{n}$. If $x_0 \in \R^{n}$ is an initial guess and $r_0 = b - A x_0$ is the initial residual, then at every step $\ell = 1, 2, \ldots, \mu$, CG produces a unique approximation \citep{YoussSaad}
\begin{equation}
    x_\ell \in x_0 + \mathcal{K}_\ell(A, r_0) \quad \text{such that} \quad r_\ell \perp \mathcal{K}_\ell(A, r_0),\label{eq:projProcess}
\end{equation}
where $\mathcal{K}_\ell(A, r_0):= \text{span}\{r_0, A r_0, \ldots, A^{\ell-1} r_0\}$ is the $\ell$-th Krylov subspace generated by $A$ and $r_0$ and $\mu$ is the grade of $r_0$ with respect to $A$, i.e., the maximum dimension of the Krylov subspace generated by $A$ and $r_0$ \citep{YoussSaad}. In exact arithmetic $x_\mu = x^\ast$, where $x^\ast$ is the exact solution. 
The most popular and efficient implementation of~\eqref{eq:projProcess} is the original formulation of \citet{Hestenes1952MethodsOC}, which recursively updates coupled 2-term recurrences for $x_{\ell+1}$, the residual $r_{\ell+1} := b - A x_{\ell+1}$, and the search direction ${p}_{\ell+1}$. \Cref{Algo:CG} states the complete procedure. A common stopping criterion is based on sufficient decrease of the relative residual norm $\|r_\ell\|_2/\|r_0\|_2$. However, in certain practical implementations, such as data assimilation, a fixed number of iterations is used as a stopping criterion due to budget constraints.
The preconditioned CG (PCG) variant is also presented, together with its companion formulation, \Cref{Algo:PCG}, which will be detailed in \Cref{subsec:Properties of a good preconditioner}.

\noindent
\begin{minipage}[t]{0.47\linewidth}
\begin{algorithm}[H]
\caption{CG\vphantom{PCG}}
\label{Algo:CG}
\begin{algorithmic}[1]
\STATE $r_0 = b - Ax_0$                               $\vphantom{\hat{r}_0 = b - A \hat{x}_0}$
\STATE $\vphantom{z_0 = F \hat{r}_0}$
\STATE $\rho_0 = r_0^\top r_0$                   $\vphantom{\hat{p}_0 = z_0}$
\STATE $p_0 = r_0$                               $\vphantom{\hat{p}_0 = z_0}$
\FOR{$\ell = 0, 1, \ldots$}
    \STATE $q_\ell = Ap_\ell$                               $\vphantom{\hat{q}_\ell = A \hat{p}_\ell}$
    \STATE $\alpha_\ell = \rho_\ell / (q_\ell^\top p_\ell)$       $\vphantom{\hat{\alpha}_\ell = \hat{\rho}_\ell / (\hat{q}_\ell^\top \hat{p}_\ell)}$
    \STATE $x_{\ell+1} = x_\ell + \alpha_\ell p_\ell$             $\vphantom{\hat{x}_{\ell+1} = \hat{x}_\ell + \hat{\alpha}_\ell \hat{p}_\ell}$
    \STATE $r_{\ell+1} = r_\ell - \alpha_\ell q_\ell$             $\vphantom{\hat{r}_{\ell+1} = \hat{r}_\ell - \hat{\alpha}_\ell \hat{q}_\ell}$                                           \STATE $\vphantom{z_{\ell+1} = F \hat{r}_{\ell+1}}$
    \STATE $\rho_{\ell+1} = r_{\ell+1}^\top r_{\ell+1}$   $\vphantom{\hat{\rho}_{\ell+1} = \hat{r}_{\ell+1}^\top z_{\ell+1}}$
    \STATE $\beta_{\ell+1} = \rho_{\ell+1} / \rho_\ell$        $\vphantom{\hat{\beta}_{\ell+1} = \hat{\rho}_{\ell+1} / \hat{\rho}_\ell}$    \STATE $p_{\ell+1} = r_{\ell+1} + \beta_{\ell+1} p_\ell$ $\vphantom{\hat{p}_{\ell+1} = z_{\ell+1} + \hat{\beta}_{\ell+1} \hat{p}_\ell}$
\ENDFOR
\end{algorithmic}
\end{algorithm}
\end{minipage}
\hfill
\begin{minipage}[t]{0.47\linewidth}
\begin{algorithm}[H]
\caption{PCG}
\label{Algo:PCG}
\begin{algorithmic}[1]
\STATE $\hat{r}_0 = b - A \hat{x}_0$ \quad\quad // $\hat{x}_0 = x_0$
\STATE $z_0 = F \hat{r}_0$
\STATE $\hat{\rho}_0 = \hat{r}_0^\top z_0$
\STATE $\hat{p}_0 = z_0$
\FOR{$\ell = 0, 1, \ldots$}
    \STATE $\hat{q}_\ell = A \hat{p}_\ell$
    \STATE $\hat{\alpha}_\ell = \hat{\rho}_\ell / (\hat{q}_\ell^\top \hat{p}_\ell)$
    \STATE $\hat{x}_{\ell+1} = \hat{x}_\ell + \hat{\alpha}_\ell \hat{p}_\ell$ 
    \STATE $\hat{r}_{\ell+1} = \hat{r}_\ell - \hat{\alpha}_\ell \hat{q}_\ell$
    \STATE $z_{\ell+1} = F \hat{r}_{\ell+1}$
    \STATE $\hat{\rho}_{\ell+1} = \hat{r}_{\ell+1}^\top z_{\ell+1}$
    \STATE $\hat{\beta}_{\ell+1} = \hat{\rho}_{\ell+1} / \hat{\rho}_\ell$
    \STATE $\hat{p}_{\ell+1} = z_{\ell+1} + \hat{\beta}_{\ell+1} \hat{p}_\ell$
\ENDFOR
\end{algorithmic}
\end{algorithm}
\end{minipage}

\subsection{Convergence properties of CG}
 
The approximation~\eqref{eq:projProcess} minimizes the error in energy norm:
\begin{equation}
\|x^\ast - x_\ell\|_{A}^2 = \min_{p\in \mathbb{P}_\ell(0)} \|p(A)(x^\ast - x_0)\|_{A}^2 = \min_{p\in \mathbb{P}_\ell(0)} \, \sum_{i=1}^n p(\lambda_i)^2 \frac{\eta_i^2 }{\lambda_i}, \label{eq:MiniProcess}
\end{equation}
where $\eta_i = s_i^\top r_0$ and $\mathbb{P}_\ell(0)$ is the set of polynomials of degree at most $\ell$ with value $1$ at zero \citep{YoussSaad}. Thus, at each iteration, before reaching the solution, CG solves a certain weighted polynomial approximation problem over the discrete set~$\{\lambda_1, \ldots, \lambda_n\}$.
Moreover, if $z_1^{(\ell)}, \ldots, z_\ell^{(\ell)}$ are the $\ell$ roots of the unique solution $p_\ell$ to~\eqref{eq:MiniProcess} \citep{gratton2014differentiating, VanDer},
\begin{equation}
\|x^\ast - x_\ell\|_{A}^2 = \sum_{i=1}^n p_\ell(\lambda_i)^2 \frac{\eta_i^2}{\lambda_i} = \sum_{i=1}^n \prod_{j=1}^\ell \left(1 - \frac{\lambda_i}{z_j^{(\ell)}}\right)^2 \frac{\eta_i^2}{\lambda_i}. \label{eq:ErrorRitz}
\end{equation}
The $z_j^{(\ell)}$ are the \emph{Ritz values} \citep{VanDer}. From~\eqref{eq:ErrorRitz}, if $z_j^{(\ell)}$ is close to a $\lambda_i$, we expect a significant reduction in the error in energy norm.
Based on the above, \citet{VanDer} explain the rate of convergence of CG in terms of the convergence of the Ritz values to eigenvalues of \(A\). 

The widely stated convergence bound  for CG
\begin{equation}\label{eq:overestimate}
    \frac{\|x^\ast - x_\ell\|_A}{\|x^\ast - x_0\|_A} \leq 2 \left(\frac{\sqrt{\kappa(A)} - 1}{\sqrt{\kappa(A)} + 1} \right)^\ell,
\end{equation}
where $\kappa(A) := \lambda_1 / \lambda_n$ is the condition number of \(A\). While~\eqref{eq:overestimate} provide the worst-case behavior of CG, the convergence properties may vary significantly from the worst case for a specific right hand side~\citep{beckermann2002superlinear}.
\subsection{Properties of a good preconditioner}\label{subsec:Properties of a good preconditioner}
In many practical applications, a preconditioner is essential for accelerating the convergence of CG \citep{Preconditionning1, wathen2015preconditioning}. Assume that a preconditioner $F = U U^\top \in \mathbb{R}^{n \times n}$ is available in factored form, where $U$ is non singular, and consider the system with split preconditioner
\begin{equation}
    U^\top A U y = U^\top b, \label{eq:SplitPrecSys}
\end{equation}
whose matrix is also SPD.
System~\eqref{eq:SplitPrecSys} can then be solved with CG. The latter updates estimate $y_{\ell}$ that can be used to recover $\hat{x}_\ell := U y_\ell$.
\Cref{Algo:PCG}, the PCG method, is equivalent to the procedure just described, but only involves products with \(F\) and does not assume knowledge of \(U\) \citep[p.532]{GoVa13}.
PCG updates $\hat{x}_\ell$ directly as an approximate solution in $\hat{x}_0 + U\mathcal{K}_{\ell}(U^\top A U , U^\top r_0)$, with $\hat{x}_0 = x_0$, as in~\eqref{eq:MiniProcess}, $\hat{x}_\ell$ minimizes the energy norm 
\begin{equation}
\lVert x^\ast - \hat{x}_\ell \rVert_{A}^2 = \min_{p \in \mathbb{P}_\ell(0)} \lVert U p (U^\top A U) U^{-1} \left( x^\ast - x_0\right) \rVert_{A}^2.\label{eq:error_energy}    
\end{equation}

Although there is no general method for building a good preconditioner \citep{Preconditionning1, wathen2015preconditioning}, leveraging the convergence properties of CG on~\eqref{eq:error_energy} often leads to the following criteria: (i) $F$ should approximate the inverse of $A$, (ii) $F$ should be cheap to apply, (iii) $\kappa(U^\top A U)$ should be smaller than $ \kappa(A)$, and (iv) $U^\top A U$ should have a more favorable distribution of eigenvalues than \(A\). 

\section{A scaled spectral preconditioner}\label{sec:A scaled spectral preconditioner}
We focus on the spectral preconditioners~\citep{Bergamaschi2020,SenSpectLMP,LMP} derived from the spectrum of $A$ that can be expressed as
\begin{equation}\label{eq:spectral_prec}
F = I_n + S D S^\top,    
\end{equation}
where 
$D = \diag(d_1, d_2, \ldots, d_n)$ is a diagonal matrix of rank $k$, \( 1 < k < n \), such that $I_n + D$ is SPD. Preconditioner $F$ can be expressed in factorized form $F = UU^\top$, with 
\begin{equation}\label{eq:spectral_prec_square}
U =  S (I_n + D)^{1/2} S^\top,    
\end{equation}
In what follows, we denote by $\hat{x}_{\ell}(F)$ the $\ell$-th PCG iterate computed with the preconditioner $F$ defined in~\eqref{eq:spectral_prec}. We also let $\mathbf{PCG}(F)$ denote the PCG algorithm when using $F$ as a preconditioner.
The next theorem provides an explicit expression for the energy norm of the error produced by $\mathbf{PCG}(F)$.
\begin{theoremE}
  \label{thm:GeneralError}
  Consider $\mathbf{PCG}(F)$ with $F$ defined in~\eqref{eq:spectral_prec}. 
  Assume that, at iteration $\ell$, the $\mathbf{PCG}(F)$ has not yet reached the solution. Then for any $j\leq \ell$,
\begin{equation}\label{eq:ErrorSpectralLMP}
    \|x^\ast - \hat{x}_{j}(F)\|_{A}^2 = \sum_{i=1}^{n} \frac{\eta_i^2}{\lambda_i} \hat{p}_j((d_i+1) \lambda_i)^2,
\end{equation}
where $\{\hat{p}_j\} = \arg\min_{p \in \mathbb{P}_{j}(0)} \sum\limits_{i=1}^n \frac{\eta_i^2}{\lambda_i} p((d_i+1) \lambda_i)^2$.
\end{theoremE}
\begin{proofEE}
Let $U$ be defined as~\eqref{eq:spectral_prec_square}. As in~\eqref{eq:error_energy},
    \begin{align*}
        \|x^\ast - \hat{x}_{j}(F)\|_{A}^2 &= \min_{p \in \mathbb{P}_{j}(0)} \lVert U p (U^\top A U) U^{-1} \left( x^\ast - x_0\right) \rVert_{A}^2\\
        &= \min_{p \in \mathbb{P}_{j}(0)} \lVert p (U^\top A U)\left( x^\ast - x_0\right) \rVert_{A}^2=\min_{p \in \mathbb{P}_{j}(0)} \sum_{i=1}^{n} \frac{\eta_i^2}{\lambda_i} p((d_i+1) \lambda_i)^2,
    \end{align*}
where we use the identity $U p(U^\top A U) U^{-1} = U U^{-1} p(U^\top A U)$, since both \( A \) and \( U \) are diagonalizable in the same eigenbasis.
\end{proofEE}

In the spectral LMP \citep{TshiGratWeavSart08} for data assimilation problems, the coefficients $\{d_i\}$ are chosen so that the $k$ largest eigenvalues $\{\lambda_i\}$ are mapped into an existing cluster, typically located at $1$, while the remaining eigenvalues remain unchanged. Here, we ask whether a better choice for $d_i$ can be made.

\subsection{An optimal choice for \boldmath{$D$}}\label{subsec:optimal_F}

We wish to find $D^\ast=\diag(d_1^\ast,\ldots,d_n^\ast)$ as a solution of 
\begin{equation}\label{eq:Prec_Opt}
    \begin{aligned}
    \min_{(d_1,\ldots,d_n) \in \R^n}\quad &\tfrac{1}{2} \|x^\ast - \hat{x}_{\ell}(F) \|_{A}^2\\
        \text{s.t.} \quad\quad  &F = I_n + S D  S^\top \quad \mbox{with}~D=\diag(d_1,\ldots,d_n),\\
        & I_n + D  \text{ is positive-definite}, \\
        & \text{rank}(D) \leq k.
    \end{aligned}
\end{equation}
Problem~\eqref{eq:Prec_Opt} defines the best low-rank update of the identity  matrix in the directions of the eigenvectors to minimize the error in energy norm at iteration $\ell$. 

In~\Cref{th:PrecOptimale}, we will show that the entries of a particular class of solutions can be constructed using \(k\) eigenvalues of \(A\) and a scaling factor \(\theta^\ast\), which is a root of a specific polynomial defined later. The choice of eigenvalues used in this construction depends on the properties of \(A\), \(b\), and \(\ell\).
To identify the appropriate \(k\) eigenvalues, we consider the set
\begin{equation*}
   \Pi_k = \{\pi_k \subset \{1,2,\ldots,n\} \mid |\pi_k| = k \}.
\end{equation*}
For any  \(\pi_k \in\Pi_k\), the indices in \(\pi_k\) determine the \(k\) nonzero 
diagonal elements of \(D\), and thus define a preconditioner \(F\). 
The complementary set of \(\pi_k\) will be denoted \(\bar{\pi}_k\). 

To prove~\Cref{th:PrecOptimale}, we first present a lemma on the existence and unicity of the polynomial minimizing,
\begin{equation}\label{poly_existance}\displaystyle \min_{p \in \mathbb{P}_{\ell}(0)}  \sum_{i \in \bar \pi_k} \frac{\eta_i^2}{\lambda_{i}} p(\lambda_{i})^2, 
\end{equation}
which is important for constructing the solution in~\Cref{th:PrecOptimale}.
\begin{lemmaE}\label{lm:unicity_p}
 Consider $\mathbf{PCG}(F)$ with $F$ feasible for~\eqref{eq:Prec_Opt}. 
  Assume that, at iteration $\ell$, the $\mathbf{PCG}(F)$ has not yet reached the solution.
 Then, for any $\pi_k \in \Pi_k$, there
exists a unique solution to~\eqref{poly_existance} which has only positive roots. 
\end{lemmaE}
\begin{proofEE}
    Let us define the SPD matrix $\Sigma := \diag((\lambda_i)_{i\in \bar\pi_k})$ and $\eta := [\eta_i]_{i\in\bar\pi_k}^\top$. Applying CG to the linear system 
  $\Sigma\;y = \eta$ ,
with the initial guess $y_0 = 0$, implicitly solves
\begin{equation}\label{eq:pltilde}
\arg \min_{p \in \mathbb{P}_{j}(0)} 
\sum_{i \in \bar \pi_k}
\frac{\eta_i^2}{\lambda_{i}} p(\lambda_{i})^2    
\end{equation}
at the $j$-th iteration, with $j \leq \widetilde{\ell}$  and $\widetilde{\ell}$ being the grade of $\eta$ with respect to $\Sigma$. Thus, to prove the existence and uniqueness of the polynomial minimizing~\eqref{poly_existance},
it suffices to show that $\widetilde{\ell}> \ell$, since CG produces the unique polynomial minimizing~\eqref{eq:pltilde}.

By contradiction, assume that $\widetilde{\ell} \leq \ell$. We first consider the case $\widetilde{\ell} \geq 1$; the particular case $\widetilde{\ell} = 0$ will be handled separately. Let $\{p_{\widetilde{\ell}}\} = \arg\min_{p\in\mathbb{P}_{\widetilde{\ell}}(0)} \sum_{i \in \bar{\pi}_k} 
\frac{\eta_i^2}{\lambda_i} \, p(\lambda_i)^2$,
and $\sigma$ be one of its positive roots (a Ritz value at iteration $\widetilde{\ell}$ when solving $\Sigma\;y=\eta$). Define a preconditioner $\widetilde{F}$ with $d_i = \sigma/\lambda_i -1$ for $i \in \pi_k$, and $d_i = 0$ for $i \in \bar{\pi}_k$.  Then,
\begin{align*}
\min_{p\in\mathbb{P}_{\widetilde{\ell}}(0)} \sum_{i\in \bar\pi_k} \frac{\eta_i^2}{\lambda_i} {p}(\lambda_i)^2
    &\leq \min_{p\in\mathbb{P}_{\widetilde{\ell}}(0)} \left [ \sum_{i\in \pi_k}\frac{\eta_i^2}{\lambda_i} {p}(\sigma)^2 
      + \sum_{i\in \bar\pi_k} \frac{\eta_i^2}{\lambda_i} {p}(\lambda_i)^2 \right ] \\
    & = \|x^\ast - \hat{x}_{\widetilde{\ell}}(\widetilde{F})\|_{A}^2\\
    &\leq \sum_{i\in \pi_k} \frac{\eta_i^2}{\lambda_i} {p_{\widetilde{\ell}}}(\sigma)^2 
       + \sum_{i\in \bar\pi_k} \frac{\eta_i^2}{\lambda_i} p_{\widetilde{\ell}}(\lambda_i)^2 
       = \min\limits_{p\in\mathbb{P}_{\widetilde{\ell}}(0)}  \sum_{i\in \bar\pi_k} \frac{\eta_i^2}{\lambda_i} {p}(\lambda_i)^2.
\end{align*}
As a result, $\|x^\ast - \hat{x}_{\widetilde{\ell}}(\widetilde{F})\|_{A}^2 
= \min\limits_{p\in\mathbb{P}_{\widetilde{\ell}}(0)}  \sum_{i\in \bar\pi_k}\frac{\eta_i^2}{\lambda_i} {p}(\lambda_i)^2 = 0$,
since CG applied to the linear system   $\Sigma\;y = \eta$ terminates at step $\widetilde{\ell}$. Since $\widetilde{\ell} \leq \ell$, the last equality contradicts the assumption that $\hat{x}_{{\ell}}(\widetilde{F})\neq x^\ast$. 

We now consider the case $\widetilde{\ell} = 0$, which occurs only if $\eta = 0$. 
Define the preconditioner $\widetilde{F}$ with $d_i = \sigma/\lambda_i -1$ for $i \in \pi_k$, and $d_i = 0$ for $i \in \bar{\pi}_k$ with $\sigma > 0$. 
From \Cref{thm:GeneralError}, and since $\eta_i=0$ for $i\in\bar\pi_k$,
\begin{equation*}
\begin{aligned}  
\|x^\ast - \hat{x}_{1}(\widetilde{F})\|_{A}^2 
    &= \min\limits_{p\in\mathbb{P}_1(0)}\sum_{i\in \pi_k} \frac{\eta_i^2}{\lambda_i} {p}(\sigma)^2
       + \sum_{i\in \bar\pi_k} \frac{\eta_i^2}{\lambda_i} {p}(\lambda_i)^2 \leq \sum_{i\in \pi_k}\frac{\eta_i^2}{\lambda_i} \hat{p}_1 (\sigma) = 0,
\end{aligned}    
\end{equation*}
where $\hat{p}_1(\lambda) = 1 - \frac{\lambda}{\sigma}$. This again contradicts the assumption that $\hat{x}_{1}(\widetilde{F})\neq x^\ast$.
Finally, we conclude that $\widetilde{\ell} > \ell$.

Since the roots of the polynomial minimizing~\eqref{poly_existance} correspond to Ritz values generated by CG applied to   $  \Sigma\;y = \eta$, all the roots are positive.
\end{proofEE}

In the rest of the paper, we  assume that $\ell\ge 1$ and that at iteration $\ell$,  $\mathbf{PCG}(F)$ has not yet reached the solution with $F$ feasible for~\eqref{eq:Prec_Opt}.

\begin{theoremE}\label{th:PrecOptimale}
Let $
    \displaystyle
\pi_k^{\ast}\in  \arg\min_{\pi_k \in\Pi_k} 
\left \{ 
\min_{p \in \mathbb{P}_{\ell}(0)} 
\sum_{i\in \bar \pi_k} 
\frac{\eta_i^2}{\lambda_{i}} p(\lambda_{i})^2 \right\}$, and   $\hat{p}_{\ell} \in \mathbb{P}_{\ell}(0)$, with a positive root $\theta^*
>0$ and
such that $
\displaystyle \{\hat{p}_{\ell}\} = \arg \min_{p \in \mathbb{P}_{\ell}(0)} 
\sum_{i \in \bar \pi^*_k}
\frac{\eta_i^2}{\lambda_{i}} p(\lambda_{i})^2. 
$

Then, $d^\ast := (d_1^{\ast},\ldots,d_n^{\ast})$, where
$
d_i^\ast :=
\begin{cases}
    \frac{\theta^\ast}{\lambda_i} - 1   , & \text{if } i \in \pi^{*}_k, \\
    0, & \text{if } i \in \bar\pi^{*}_k, 
\end{cases},
$ solves~\eqref{eq:Prec_Opt}.

In addition,
\begin{equation}\label{eq:erroroptimalpi}
    \|x^\ast - \hat{x}_{\ell} (F^\ast) \|_{A}^2 = \sum_{i\in \bar \pi_k^\ast} 
\frac{\eta_i^2}{\lambda_{i}} \hat{p}_{\ell}(\lambda_{i})^2, 
\end{equation}
where
$F^\ast:=I_n + S D^* S^\top$ with $D^*=\diag(d^*_1,\ldots,d^*_n)$.
\end{theoremE}

\begin{proofEE}
First, the existence of such a $\pi_k^\ast$ follows from the fact that the set $\Pi_k$ is finite and that, for each $\pi_k$, the inner minimization over $\mathbb{P}_\ell(0)$ is well defined.
 Let $F = I_n + S D S^\top \in \mathbb{R}^{n \times n}$ with $D=\text{diag}(d_1,\ldots,d_n)$ where $(d_1,\ldots,d_n)$ is
    feasible for~\eqref{eq:Prec_Opt}.  
Let 
$
\mathcal{Z} = \{i \in \{1,\ldots,n\}~|~ d_i = 0\}$ and $ \mathcal{N} = \{i \in \{1,\ldots,n\} ~|~d_i \neq 0\}
$. Then, using \Cref{thm:GeneralError},
\begin{equation}
\begin{aligned}
\|x^\ast - \hat{x}_{\ell}(F)\|_{A}^2 
&=\min_{p \in \mathbb{P}_{\ell}(0)} \sum_{i\in \mathcal{N}} \frac{\eta_i^2}{\lambda_i} p((d_i+1) \lambda_i)^2 + \sum_{i\in \mathcal{Z}} \frac{\eta_i^2}{\lambda_i} p(\lambda_i)^2  \\ 
&\geq \min_{p \in \mathbb{P}_{\ell}(0)} \sum_{i\in \mathcal{Z}} \frac{\eta_i^2}{\lambda_i} p(\lambda_i)^2\label{eq:lastineq}.
\end{aligned}
\end{equation}
Since \( D \) has rank at most \( k \),  $|\mathcal{Z}|\geq n -k$.

Let $\mathcal{\bar \pi}_k \subseteq \mathcal{Z}$ such that $|\mathcal{\bar \pi}_k| = n-k$.  Then, using~\eqref{eq:lastineq}, the optimality of set $\pi_k^\ast$,  and the fact that $\hat{p}_{\ell}(\theta^\ast)=0$, we get
\begin{eqnarray*}
    \|x^\ast - \hat{x}_{\ell}(F)\|_{A}^2 
\geq \min\limits_{p\in\mathbb{P}_{\ell}(0)}\sum_{i\in \mathcal{\bar \pi}_k} \frac{\eta_i^2}{\lambda_i} p(\lambda_i)^2 &  \geq &\sum_{i \in \bar \pi_k^\ast}
    \frac{\eta_{i}^2}{\lambda_{i}} \hat{p}_{\ell}(\lambda_{i})^2\\
&
      =& \sum_{i \in \pi^{*}_k}
    \frac{\eta_{i}^2}{\lambda_{i}} \hat{p}_{\ell}(\theta^\ast)^2 + \sum_{i \in \bar \pi_k^\ast} 
    \frac{\eta_{i}^2}{\lambda_{i}} \hat{p}_{\ell}(\lambda_{i})^2  \\
    &\geq &\min\limits_{p\in\mathbb{P}_{\ell}(0)}\sum_{i \in \pi^{*}_k}
    \frac{\eta_{i}^2}{\lambda_{i}} p(\theta^\ast)^2 + \sum_{i \in \bar \pi_k^\ast} 
    \frac{\eta_{i}^2}{\lambda_{i}} p(\lambda_{i})^2 \\
 &
      =& \min\limits_{p\in\mathbb{P}_{\ell}(0)}\sum_{i=1}^{n} \frac{\eta_i^2}{\lambda_i} p((d_i^\ast+1) \lambda_i)^2 \\
 &      
      =&   \|x^\ast - \hat{x}_{\ell}(F^\ast)\|_{A}^2. \hspace*{2em} (\text{\small{by \Cref{thm:GeneralError}}})
\end{eqnarray*}
Hence, 
\begin{equation}\label{eq:lowerboundopti}
\begin{aligned}
\|x^\ast - \hat{x}_{\ell}(F)\|_{A}^2 & \ge      \sum_{i \in \bar \pi_k^\ast}
    \frac{\eta_{i}^2}{\lambda_{i}} \hat{p}_{\ell}(\lambda_{i})^2\ge \|x^\ast - \hat{x}_{\ell}(F^\ast)\|_{A}^2.
\end{aligned}    
\end{equation}
Since~\eqref{eq:lowerboundopti} holds for any feasible $F$, ${d^*}$ solves~\eqref{eq:Prec_Opt}.
Moreover, as~\eqref{eq:lowerboundopti} also holds for $F^*$, $\|x^\ast - \hat{x}_{\ell}(F^\ast)\|_{A}^2 = \sum_{i \in \bar \pi_k^\ast}
    \frac{\eta_{i}^2}{\lambda_{i}} \hat{p}_{\ell}(\lambda_{i})^2$.
\end{proofEE}


The following corollary provides an explicit expression for $F^\ast$ together with a square root decomposition.
\begin{corollary}
   Solution of~\eqref{eq:Prec_Opt} can be expressed as
\begin{equation}
\label{eq:U_perm}
    F^\ast 
    := F({\theta^\ast, \pi^{*}_k})= I_n + \sum\limits_{i\in \pi^{*}_k }\left(\frac{\theta^\ast}{\lambda_i}-1\right)s_{i}s_{i}^\top= U^\ast {U^{\ast}}^\top,
\end{equation}
where $U^\ast= {U^{\ast}}^\top := I_n + \sum\limits_{i\in \pi^{*}_k }\left(\sqrt{\frac{\theta^\ast}{\lambda_i}}-1\right)s_{i}s_{i}^\top $, and $\pi^{*}_k$ and $\theta^\ast$ are defined in \Cref{th:PrecOptimale}. 
\end{corollary}

Note that \( F^\ast \) maps eigenvalues \( \{\lambda_i\}_{i\in\pi^{*}_k} \) at $\theta^\ast$, while leaving the rest of the spectrum unchanged. Specifically,  
\[
F^\ast A 
= \theta^\ast \sum_{i \in \pi_k^\ast} s_i s_i^\top
+ \sum_{i \in \bar\pi_k^\ast} \lambda_i s_i s_i^\top.
\]
Hence, the spectrum is $\{\theta^\ast \text{ (with multiplicity } k)\} \cup \{\lambda_i \mid i \in \bar\pi_k^\ast\}$.
The following corollary states that {$\mathbf{PCG}(F^\ast)$}
guarantees a reduction in the energy norm of the error compared to the unpreconditioned case. In fact, since $D=0$ is feasible for~\eqref{eq:Prec_Opt}, the {$\mathbf{PCG}(I_n)$} iterates coincide with those of CG. Hence, by the optimality of $F^\ast$ for~\eqref{eq:Prec_Opt}, the result follows directly.
\begin{corollary}
    {$\mathbf{PCG}(F^\ast)$} leads to an improved reduction in the energy norm of the error compared to the unpreconditioned case: 
\begin{equation}\label{eq:BetterConvergence}
    \|x^\ast - \hat{x}_{\ell}(F^\ast) \|_{A}^2 \leq \|x^\ast - x_{\ell} \|_{A}^2.
\end{equation}
\end{corollary}

Inequality~\eqref{eq:BetterConvergence} is particularly relevant for applications in which the reduction of the error in the energy norm is a primary objective and PCG is terminated prior to convergence.

The optimal set \(\pi_k^{\ast}\) and the value of \(\theta^{\ast}\) depend on the distribution of the eigenvalues \(\{\lambda_i\}\), the components of the initial residual \(\{\eta_i\}\), and \(\ell\).
Thus, computing \(\pi_k^{\ast}\) and \(\theta^{\ast}\) can be computationally infeasible in practical applications.
In the next section, we propose a practical choice of \(\pi_k^{\ast}\) obtained by minimizing the condition number; a closely related idea appears in \citep{agullo2018low} in the context of hierarchical matrices.
Practical strategies for selecting the scaling parameter, as alternatives to \(\theta^{\ast}\), are also proposed.

\subsection{A practical choice for $\pi_k^\ast$}\label{subsec:pi_k}
We first derive an upper bound on $\|x^\ast - \hat{x}_{\ell}(F^\ast) \|_{A}$ based on the condition number.
\begin{lemmaE}\label{lemma:pi_k^a}
Let $\kappa_{\bar \pi_k}:= \max \limits_{i\in \bar \pi_k}\lambda_{i}/\min \limits_{i\in\bar \pi_k}\lambda_{i}$ be defined for any $\pi_k \in \Pi_k$. Then,
\begin{align*}
\|x^\ast - \hat{x}_{\ell}(F^\ast) \|_{A} 
&\leq  
2 \left(\frac{\sqrt{\kappa_{\bar\pi_k^{a}}} - 1}{\sqrt{\kappa_{\bar\pi_k^{a}}} + 1} \right)^{\ell} 
\|x^\ast - x_0\|_{A},
\end{align*} 
where $\pi_k^{a} \in \arg\min\limits_{{\pi_k}\in\Pi_k}\kappa_{\bar \pi_k}$.
\end{lemmaE}
\begin{proofEE}
    Let $\pi_k \in \Pi_k$, then by \Cref{th:PrecOptimale}, we get
 \begin{align*}
    \|x^\ast - \hat{x}_{\ell}(F^\ast) \|_{A}^2  
&\leq \min_{p \in \mathbb{P}_{\ell}(0)} 
\sum_{i \in \bar \pi_k} 
\frac{\eta_{i}^2}{\lambda_{i}} p(\lambda_{i})^2\\
&\leq \min_{p \in \mathbb{P}_{\ell}(0)} \max_{i\in\bar \pi_k} p(\lambda_{i})^2 {\sum_{i\in\bar \pi_k} 
\frac{\eta_{i}^2}{\lambda_{i}}}\\
&\leq \min_{p \in \mathbb{P}_{\ell}(0)} \max_{i\in \bar \pi_k} p(\lambda_{i})^2 {\sum_{i=1}^{n} 
\frac{\eta_{i}^2}{\lambda_{i}}}\\
&\leq 4 \left(\frac{\sqrt{\kappa_{{\bar \pi_k}}} - 1}{\sqrt{\kappa_{{\bar \pi_k}}} + 1} \right)^{2\ell} {\sum_{i=1}^{n} 
\frac{\eta_{i}^2}{\lambda_{i}}}= 4 \left(\frac{\sqrt{\kappa_{\bar \pi_k}} - 1}{\sqrt{\kappa_{\bar \pi_k}} + 1} \right)^{2\ell} 
\|x^\ast - x_0\|_{A}^2,
\end{align*}   
where $ \kappa_{\bar \pi_k} := \max \limits_{i\in \bar \pi_k}\lambda_{i} / \min \limits_{i\in\bar \pi_k}\lambda_{i}$ defines the condition number 
associated with the set of eigenvalues, $(\lambda_i)_{i \in \bar \pi_k}$, which are not used in constructing the preconditioner.

Since, $\kappa \mapsto \left(\frac{\sqrt{\kappa}-1}{\sqrt{\kappa}+1}\right)^{2\ell}$ is increasing on $[1,+\infty[$, thus
\begin{align*}
\|x^\ast - \hat{x}_{\ell}(F^\ast) \|_{A}^2
&\leq  
4 \left(\frac{\sqrt{\kappa_{\bar\pi_k^{a}}} - 1}{\sqrt{\kappa_{\bar\pi_k^{a}}} + 1} \right)^{2\ell} 
\|x^\ast - x_0\|_{A}^2,
\end{align*}
where $\pi_k^{a} \in \arg\min\limits_{{\pi_k}\in\Pi_k}\kappa_{\bar \pi_k}$.
\end{proofEE}

The next proposition 
gives a characterization of $\pi_k^{a}$.
\begin{propositionE}
\label{co:condequiv}
Let  $\mathcal{J} = \arg\min\limits_{1 \leq j \leq k+1} \lambda_j/\lambda_{n-k+j-1}$, $j_0 \in \mathcal{J}$,  
and $\pi_k \in\Pi_k$ such that $\bar\pi_k = \{j_0, j_0+1, \ldots,j_0 +n-k-1\}$.
Then, $\pi_k \in \arg\min\limits_{\pi_k' \in\Pi_k} \kappa_{\bar\pi_k'}$.
\end{propositionE}
\begin{proofEE}
    Let \(\pi_k' \in\Pi_k\) and set \(j_0'= {\displaystyle \min_{i \in {\bar\pi}_{k}'}}~\{i\}\) and \(\bar{j}_0' ={\max\limits_{i \in \bar {\pi}_{k}'}} \{i\}\).
Using the fact that $\bar \pi_{k}^{'} \subset \{1,\ldots,n\}$ with elements ranging between \( j_0' \) and \( \bar{j}_0'\), we deduce that  \(\bar \pi_{k}' \subseteq \{j_0',j_0'+1,\ldots, \bar{j}_0'\} \).
Since $|\bar\pi_{k}'| = n-k$, we get
\begin{equation*}\label{eq:card}
    \bar{j}_0' \geq j_0' + n - k - 1 \quad \text{and} \quad j_0' \leq k+1.
\end{equation*}
Since $(\lambda_i)_{1\leq i \leq n}$ are given in decreasing order, 
$\kappa_{\bar\pi_k'} = \frac{\max_{i\in {\bar  \pi'}_k}\lambda_{i}}{\min_{i\in{\bar  \pi'}_k}\lambda_{i}}=\frac{\lambda_{j_0'}}{\lambda_{\bar{j}_0'}} \ge \frac{\lambda_{j_0'}}{\lambda_{j_0' + n - k - 1}}$.

Consider $\pi_k \in\Pi_k$ such that 
    $\bar\pi_k = \{j_0, j_0+1, \ldots,j_0 +n-k-1\}$.
Then, for all $\pi_k' \in\Pi_k$, we get
\[
\kappa_{\bar\pi_k'} \ge \frac{\lambda_{j_0'}}{\lambda_{j_0' + n - k - 1}} \ge \frac{\lambda_{j_0}}{\lambda_{j_0 + n - k - 1}} = \kappa_{\bar\pi_k}, 
\]
i.e., $\pi_{k} \in \arg\min\limits_{\pi_k' \in\Pi_k} \kappa_{\bar\pi_k'}$. 
\end{proofEE}

\Cref{co:condequiv} shows that once an index \( j_0 \in \mathcal{J} \) has been identified, an optimal \( \pi_k^{a} \) can be constructed simply by ensuring that the smallest and largest values of \( \bar \pi_k^{a} \) correspond to \( j_0 \) and \(n-k-1+j_0\), respectively. 
As a result, the preconditioner can be constructed by using \( \pi_k^{a} \) as follows
\begin{equation}
\label{eq:Ftheta}
F_{\theta}  := F(\theta, \pi_{k}^a)
\end{equation}
where $\theta > 0$, and strategies for selecting it will be presented later.

Let $j_0 \in \mathcal{J}$ such that $\bar\pi_k^a = \{j_0,j_0 + 1, \ldots,n-k+j_0-1\}$. The value of \( j_0 \) determines which part of the eigenspectrum is used in constructing the preconditioner. There are three possible cases:
\begin{enumerate}[label=\textit{Case \arabic*.}]
\item Using the largest \(k\) eigenvalues: \( j_0 = k+1 \) and \( \pi_{k}^{a} = \{1, 2, \ldots, k\} \).
\item Using the smallest \(k\) ones: \( j_0 = 1 \) and  \( \pi_{k}^{a} = \{n-k+1, n-k+2, \ldots, n\} \).
\item Using a mixture of smaller and larger ones: \( 1 < j_0 < k+1 \) and \newline \( \pi_{k}^{a} = \{1, \ldots, j_0 - 1\} \cup \{n-k + j_0, \ldots, n\} \).
\end{enumerate}


Now that $\pi_k^a$ has been set according to the application context (\textit{Cases 1, 2}, or \textit{3}), we provide a practical estimate of $\theta$ in the next section. In particular, we present four strategies for selecting $\theta$.

\section{On the choice of the scaling parameter $\theta$}\label{sec:Choices of the scalling}
\subsection{\boldmath{ \texorpdfstring{$\theta$}{theta}  } as the mid-range between \texorpdfstring{$\lambda_{j_0}$}{lambdak} and \texorpdfstring{$\lambda_{n-k+j_0 -1}$}{lambdan}}
\label{SubSec:median}
Our goal is to select $\theta$ such that the resulting {$\mathbf{PCG}(F_\theta)$} iterates yield an error comparable {to~\eqref{poly_existance} with $\pi_k = \pi_k^a$.} 
We begin by examining the connection with deflation methods \citep{DeflatedCG,DeflatedCG2}.
The deflation method, when used with the deflation subspace constructed from the eigenvectors ${(s_i})_{i\in\pi_k^a}$    corresponding to the eigenvalues 
${(\lambda_i})_{i\in\pi_k^a}$, generates iterates 
\begin{equation}\label{eq:x_l^D}
  x_\ell^{\text{D}} = \sum\limits_{i\in\pi_k^a} \frac{1}{\lambda_i}s_i s_i^\top b + P z_\ell,
\end{equation}
where $P = I_n -\sum_{i\in\pi_k^a} s_i s_i^\top$ is the orthogonal projection onto $\text{span}\{s_i \mid i\in\bar\pi_k^a\}$ and $z_\ell$ designates the iterate generated by CG when solving the  projected system~(see \Cref{prop:PA=AP})
\begin{equation}\label{eq:projected_sys}
      P  A \; z = P b 
\end{equation}
starting from $z_0 = x_0$. 


The following theorem provides the polynomial expression for $\left\| x^\ast - x_\ell^{\text{D}}\right\|_{A}^2$.

\begin{theoremE}\label{th:xast_xdef}
The energy norm of the error for $x_\ell^D$ defined in~\eqref{eq:x_l^D} is given by
    \begin{align}\label{eq:DefPoly}
    \left\| x^\ast - x_\ell^{\text{D}}\right\|_{A}^2
    & = 
    \min_{p\in \mathbb{P}_{\ell}(0)} 
    \sum_{i\in \bar\pi_k^a} \frac{\eta_i^2}{\lambda_i}{p}(\lambda_i)^2.       
\end{align}
\end{theoremE}
\begin{proofEE}
    The exact solution of $A\;x =b$ can be written as $
 x^\ast  = \sum_{i=1}^n \frac{1}{\lambda_i} s_i s_i^\top b.
$
From~\eqref{eq:x_l^D} 
\begin{equation}\label{eq:error-decomp}
  x^\ast - x_\ell^{\mathrm D}
  = \sum_{i\in \bar\pi_k^a} \frac{1}{\lambda_i} s_i s_i^\top b - P z_\ell.
\end{equation}
By \Cref{prop:reduced_system}, $z_\ell= \sum\limits_{i\in\pi_k^a} s_i s_i^\top x_0 + S_{\bar\pi_k^a} \hat{y}_\ell.$
Since $P s_i=0$ for $i\in\pi_k^a$ and $P s_i=s_i$ for $i\in\bar\pi_k^a$, we obtain $P z_\ell = S_{\bar\pi_k^a}\, \hat y_\ell$.
Inserting into \eqref{eq:error-decomp},
\[
  x^\ast - x_\ell^{\mathrm D}
  = S_{\bar\pi_k^a} \big(\Lambda_{\bar\pi_k^a}^{-1} S_{\bar\pi_{k}^{a}}^\top b - \hat y_\ell \big).
\]
Therefore, using $S_{\bar\pi_k^a}^\top A S_{\bar\pi_k^a} = \Lambda_{\bar\pi_k^a}$, we obtain, $\|x^\ast - x_\ell^{\mathrm D}\|_A^2
  = \|\Lambda_{\bar\pi_k^a}^{-1} S_{\bar\pi_k^a}^\top b - \hat y_\ell\|_{\Lambda_{\bar\pi_k^a}}^2$.

CG applied to the reduced system,
$\Lambda_{\bar\pi_k^a} \hat y = S_{\bar\pi_k^a}^\top b$, with starting vector 
$\hat y_0 = S_{\bar\pi_k^a}^\top x_0$ satisfies
\begin{align*}
     \|\Lambda_{\bar\pi_k^a}^{-1} S_{\bar\pi_k}^\top b - \hat y_\ell\|_{\Lambda_{\bar\pi_k^a}}^2
  &=  \min_{p\in\mathbb{P}_\ell(0)}
  \big\| p(\Lambda_{\bar\pi_k^a}) (\Lambda_{\bar\pi_k^a}^{-1}S_{\bar\pi_k^a}^\top b - \hat y_0)\big\|_{\Lambda_{\bar\pi_k^a}}^2\\
  &=\min_{p\in\mathbb{P}_\ell(0)}
  \big\| p(\Lambda_{\bar\pi_k^a}) \Lambda_{\bar\pi_k^a}^{-1}S_{\bar\pi_k^a}^\top ( b - S_{\bar\pi_k^a}  \Lambda_{\bar\pi_k^a} S_{\bar\pi_k^a}^\top x_0)\big\|_{\Lambda_{\bar\pi_k^a}}^2\\
  & = \min_{p\in\mathbb{P}_\ell(0)} \sum_{i\in \bar\pi_k^a}
  \frac{\eta_i^2}{\lambda_i} \, p(\lambda_i)^2.
  \tag*{\qed}
\end{align*}
\end{proofEE}

Let $\{p_\ell^{\text{D}}\} := \arg  \min_{p\in \mathbb{P}_{\ell}(0)} 
    \sum\limits_{i\in \bar\pi_k^a} \frac{\eta_i^2}{\lambda_i}{p}(\lambda_i)^2$ denote the polynomial that attains~\eqref{eq:DefPoly}.
The following theorem presents the main result of this section.
\begin{theoremE}\label{theorem:midrange}
Let $F_\theta$ be defined as~\eqref{eq:Ftheta} with $\theta >0$. Let $x_\ell^{\text{D}}$ be given by~\eqref{eq:x_l^D}. Then,
\begin{equation}\label{eq:upperboundmed}
 \left\| x^\ast - x_{\ell}^{\text{D}}\right\|_{A} \leq \left\| x^\ast - \hat{x}_{\ell}(F_\theta)\right\|_{A} \leq  \frac{\alpha(\theta)}{\theta} \left\| x^\ast - x_{\ell-1}^{\text{D}}\right\|_{A},
\end{equation}
 with 
$\alpha(\theta) = \max\left(|\lambda_{j_0} - \theta|,| \theta - \lambda_{{n-k+j_0 -1}} |\right).$
\end{theoremE}
 
\begin{proofEE}
        Let us start by proving the first inequality. Using expression~\eqref{eq:ErrorSpectralLMP} together with the definition of $F_{\theta}$, we obtain
\begin{align*}
    \left\| x^\ast - \hat{x}_{\ell}(F_\theta)\right\|_{A}^2 
    & = \min\limits_{p\in\mathbb{P}_{\ell}(0)}\sum_{i\in \pi_k^a} \frac{\eta_i^2}{\lambda_i} p(\theta)^2 + \sum_{i \in \bar\pi_k^a} \frac{\eta_i^2}{\lambda_i} p(\lambda_i)^2\\
    & \geq \min_{p\in \mathbb{P}_{\ell}(0)} 
          \sum\limits_{i \in \bar\pi_k^a} \frac{\eta_i^2}{\lambda_i}{p}(\lambda_i)^2  
          = \left\| x^\ast - x_{\ell}^{\text{D}}\right\|_{A}^2,
\end{align*}    
where the last equality follows from \Cref{th:xast_xdef}.
To prove the second equality, we define $\widetilde{p}(\lambda) = \left(1 - \frac{\lambda}{\theta}\right) p_{\ell-1}^{\text{D}}(\lambda) \in \mathbb{P}_{\ell}(0)$, where $\{p_{\ell-1}^{\text{D}}\} = \arg  \min\limits_{p\in \mathbb{P}_{\ell-1}(0)} 
    \sum\limits_{i\in \bar\pi_k^a} \frac{\eta_i^2}{\lambda_i}{p}(\lambda_i)^2$. By construction, $\widetilde{p}(\theta)=0$. Then,
\begin{align*}
    \left\| x^\ast - \hat{x}_{\ell}(F_\theta)\right\|_{A}^2 
    & = \min\limits_{p\in\mathbb{P}_{\ell}(0)}\sum_{i \in \pi_k^a} \frac{\eta_i^2}{\lambda_i} p(\theta)^2 + \sum_{i \in \bar\pi_k^a} \frac{\eta_i^2}{\lambda_i} p(\lambda_i)^2\\
    & \leq  \sum_{i \in \pi_k^a} \frac{\eta_i^2}{\lambda_i} \widetilde{p}(\theta)^2 + \sum_{i \in \bar\pi_k^a} \frac{\eta_i^2}{\lambda_i} \widetilde{p}(\lambda_i)^2= \sum_{i \in \bar\pi_k^a} \frac{\eta_i^2}{\lambda_i} p_{\ell-1}^{\text{D}}(\lambda_i)^2\left(1 - \frac{\lambda_i}{\theta}\right)^2\\
    & \leq \max_{i \in \bar\pi_k^a}\left(1 - \frac{\lambda_i}{\theta}\right)^2 \left\| x^\ast - x_{\ell-1}^{\text{D}}\right\|_{A}^2 = \frac{\alpha(\theta)^2}{\theta^2} \left\|  x^\ast- x_{\ell-1}^{\text{D}}\right\|_{A}^2.
\end{align*}  
Here we have used~\eqref{eq:DefPoly}, 
as well as the identity $\alpha(\theta) / \theta = \max_{i \in \bar\pi_k^a}(1 - \lambda_i / \theta)$, which follows from the ordering of 
the eigenvalues. 
\end{proofEE}

Note that  choosing $\theta>0$ such that $\alpha(\theta)/\theta \geq 1$  in~\eqref{eq:upperboundmed} would give a pessimistic upper bound. 
For a better bound, we select $\theta>0$ such that $\alpha(\theta)/\theta < 1$, which is equivalent to imposing $\theta > \lambda_{j_0}/2$.
The value of $\theta$ that minimizes $\alpha(\theta)/\theta$ is  $\theta_{\mathrm{m}} = (\lambda_{j_0}+\lambda_{n-k+j_0-1})/2$, for which $\alpha(\theta_{\mathrm{m}})/\theta_{\mathrm{m}} = (\lambda_{j_0} - \lambda_{n-k+j_0 -1})/(\lambda_{j_0} + \lambda_{n-k+j_0 -1})< 1$.

In practice, we do not have access to \(\lambda_{j_0}\) and \(\lambda_{n-k+j_0-1}\), so we propose practical choices of \(\theta\) such that $\alpha(\theta)/\theta < 1$, depending on which eigenvalues are used to build the preconditioner:
\begin{enumerate}[label=\textit{Case \arabic*.}]
\item Using the largest eigenvalues, i.e., \(j_0 = k+1\). If the smallest eigenvalue can be estimated. Then, practical choices are \(\theta = \bigl(\lambda_{k}+\lambda_{n}\bigr)/2\) or \(\theta = \lambda_k\).
\item Using the smallest eigenvalues, i.e., \(j_0 = 1\). If the largest eigenvalue can be estimated. Then, practical choices are \(\theta = \bigl(\lambda_{1}+\lambda_{n-k+1}\bigr)/2\) or \(\theta = \lambda_{1}\).
\item Using a  mixture of largest and smallest eigenvalues, i.e., \(1 < j_0 < k+1\). Then, two practical choices are \(\theta = \bigl(\lambda_{j_0-1}+\lambda_{\,n-k+j_0}\bigr)/2\) or \(\theta = \lambda_{j_0-1}\).
\end{enumerate}

\subsection{\boldmath{\texorpdfstring{$\theta$}{theta}} with respect to the initial iteration}\label{SubSec:optimum}

In this section, we investigate the choice of $\theta$ as the root of the polynomial $p_1^{D}$ given by
\begin{equation}\label{eq:p_1}
\{p_1^{D}\} = \arg\min\limits_{p\in \mathbb{P}_1(0)} 
\sum\limits_{i \in \bar\pi_k^a}\frac{\eta_i^2}{\lambda_i} p(\lambda_i)^2.
\end{equation}

\begin{theoremE}\label{th:ThetaOpt}
Let $\theta_1$ be the root of $p_1^{D}$. 
Then, $\|x^\ast - \hat{x}_1(F_{\theta_1})\|_A^2 = \|x^\ast - x_1^{D}\|_{A}^2$.
\end{theoremE}
\begin{proofEE}
    By using the first inequality from \Cref{theorem:midrange} for the particular choice $\ell=1$,
\begin{align*}
    \|x^\ast - x_1^{D}\|_{A}^2 
    &\leq \|x^\ast - \hat{x}_1(F_{\theta_1})\|_{A}^2\\
    &= \min\limits_{p\in \mathbb{P}_1(0)} 
       \sum\limits_{i \in \pi_k^a}\frac{\eta_i^2}{\lambda_i}p(\theta_1)^2
       + \sum\limits_{i \in \bar\pi_k^a}\frac{\eta_i^2}{\lambda_i} p(\lambda_i)^2\\
    &\leq \sum\limits_{i \in \pi_k^a}\frac{\eta_i^2}{\lambda_i} p_1^{D}(\theta_1)^2 
       + \sum\limits_{i \in \bar\pi_k^a}\frac{\eta_i^2}{\lambda_i} p_1^{D}(\lambda_i)^2 
       = \|x^\ast - x_1^{D}\|_{A}^2.
\end{align*}
\end{proofEE}

An explicit form of $\theta_1$
is provided in the following theorem.
\begin{theoremE}\label{th:theta_1}
$\displaystyle{\theta_1 = \frac{r_0^\top A r_0 - \sum_{i \in \pi_k^a} \lambda_i(s_i^\top r_0)^2}{r_0^\top r_0 - \sum_{i \in \pi_k^a}(s_i^\top r_0)^2}}$ is the root of $p_1^{D}$.
\end{theoremE}

\begin{proofEE}
   From \Cref{th:ThetaOpt}, 
\[\|x^\ast - \hat{x}_1(F_{\theta_1})\|_A^2  = \min\limits_{p\in \mathbb{P}_1(0)} 
\sum\limits_{i \in \bar\pi_k^a}\frac{\eta_i^2}{\lambda_i} p(\lambda_i)^2.
\]
This minimization problem is equivalent to performing one iteration of CG on $\Sigma\,y = \eta$, where the SPD matrix is given by $\Sigma = \diag((\lambda_i)_{i\in \bar\pi_k^a})$ and $\eta = [\eta_i]_{i\in\bar\pi_k^a}^\top$, using the initial guess $y_0 = 0$. The corresponding Ritz value after this first iteration is~\cite[p.194]{YoussSaad},
\[
\theta_1= \frac{\eta^\top \Sigma\eta}{\eta^\top\eta} =\frac{\sum_{i \in \bar\pi_k^a}\eta_i^2\,\lambda_i}{\sum_{i \in \bar\pi_k^a}\eta_i^2}.
\]
$\theta_1$ can also be written in terms of $\pi_k^a$, i.e.
\begin{align*}
\theta_1=\frac{\sum\limits_{i =1}^n\lambda_i\,(s_i^\top r_0)^2-\sum\limits_{i \in \pi_k^a}\lambda_i\,(s_i^\top r_0)^2}
              {\sum\limits_{i =1}^n(s_i^\top r_0)^2-\sum\limits_{i \in \pi_k^a}(s_i^\top r_0)^2}=\frac{r_0^\top A r_0-\sum\limits_{i \in \pi_k^a}\lambda_i\,(s_i^\top r_0)^2}
              {r_0^\top r_0-\sum\limits_{i \in \pi_k^a}(s_i^\top r_0)^2}.    
\end{align*}
\end{proofEE}

In the next subsection, we investigate a choice of $\theta$ that yields a smaller error than the unpreconditioned iterate at every {$\mathbf{PCG}(F_\theta)$} iterate. 

\subsection{\boldmath{\boldmath{\texorpdfstring{$\theta$}{theta}}} providing lower error in energy norm}
\label{SubSec:Case1}
In this section,  we focus on the first case where $j_0 = k+1$, i.e., the preconditioner is constructed using the largest k eigenvalues. 
We present the analysis only for this case. The other cases can be treated in a similar manner. 

We now characterize the values of $\theta$ for which the preconditioned iterates achieve a smaller error than the unpreconditioned ones. Specifically, we focus on the interval $\theta \in [\lambda_{k+1}, \lambda_k]$ and show that, for such choices, there exists a polynomial that promotes favorable {$\mathbf{PCG}(F_\theta)$ convergence}.

\begin{lemmaE}\label{lemma:Poly_comp}
For any $\theta \in [\lambda_{k+1}, \lambda_k]$, and any polynomial $p$ of degree $\ell$ such that $p(0) = 1$ and whose roots all lie in $[\lambda_n, \lambda_1]$, there exists a polynomial $\hat{p}$ of degree $\ell$ such that $\hat{p}(0) = 1$ and
\begin{align*}
    \lvert \hat{p}(\theta) \rvert & \leq \lvert p(\lambda_i) \rvert, \quad i = 1, \ldots, k\ \\
    \lvert \hat{p}(\lambda_i) \rvert & \leq \lvert p(\lambda_i) \rvert, \quad i = k+1, \ldots, n.
\end{align*}
\end{lemmaE}
\begin{proofEE}
        Let us denote $(\mu_j)_{1 \leq j \leq \ell}$ the roots of the polynomial $p$ given in decreasing order, so $p(\lambda) = \prod_{i = 1}^{\ell}\left( 1 - \frac{\lambda}{\mu_i} \right)$ for any $\lambda \geq 0$. Three cases may occur:

        \underline{Case 1:} For all $j \in \{1, \ldots, \ell\}$, $\mu_j < \theta$. We choose $\hat{p}(\lambda) = p(\lambda)$. Then for $i \in \{k+1, \ldots, n\}$, we have $\lvert \hat{p}(\lambda_i) \rvert = \lvert p(\lambda_i) \rvert$. For $i \in \{1, \ldots, k\}$, using the property that $\mu_j < \theta \leq \lambda_i$, we obtain 
        $$1 - \frac{\lambda_i}{\mu_j} \leq 1 - \frac{\theta}{\mu_j} \leq 0.$$ 
        Thus, we have $\lvert 1 - \frac{\theta}{\mu_j} \rvert \leq \lvert 1 - \frac{\lambda_i}{\mu_j} \rvert$, and consequently $\lvert \hat{p}(\theta) \rvert \leq \lvert p(\lambda_i) \rvert$.

           \underline{Case 2:} For all $j \in \{1, \ldots, \ell\}$, $\theta \leq \mu_j$. We choose $\hat{p}(\lambda) = \prod_{j = 1}^{\ell}\left( 1 - \frac{\lambda}{\theta} \right)= \left( 1 - \frac{\lambda}{\theta} \right)^{\ell}$. Then simply for $i \in \{1, \ldots, k\}$, $\lvert \hat{p}(\theta) \rvert = 0 \leq \lvert p(\lambda_i) \rvert$. For $i \in \{k+1, \ldots, n\}$, using the property $\lambda_{k+1} \leq \theta \leq \mu_j$, we obtain        $$0 \leq 1 - \frac{\lambda_i}{\lambda_{k+1}} \leq 1 - \frac{\lambda_i}{\theta} \leq 1 - \frac{\lambda_i}{\mu_j}.$$        Therefore, for $i = k+1, \ldots, n$, $\lvert \hat{p}(\lambda_i) \rvert \leq \lvert p(\lambda_i) \rvert$.
        
       \underline{Case 3:}  Let $s \in \{1, \ldots, \ell-1\}$ such that for $j = 1, \ldots, s$, $\theta \leq \mu_j \leq \lambda_1$, and for $j = s+1, \ldots,\ell$, $\lambda_n \leq \mu_j < \theta$. We choose 
        $$\hat{p}(\lambda) = \prod_{j = 1}^{s}\left( 1 - \frac{\lambda}{\theta} \right) \prod_{j = s+1}^{\ell}\left( 1 - \frac{\lambda}{\mu_j} \right)= \left( 1 - \frac{\lambda}{\theta} \right)^s \prod_{j = s+1}^{\ell}\left( 1 - \frac{\lambda}{\mu_j} \right) .$$
        We have $\hat{p} (\theta) = 0$, so $\lvert \hat{p}(\theta) \rvert \leq \lvert p(\lambda_i) \rvert$ for $i \in \{1, \ldots, k\}$. For $i \in \{k+1, \ldots, n\}$ and  $j \in \{1, \ldots, s\}$, we have 
        $$0 \leq 1 - \frac{\lambda_i}{\lambda_{k+1}} \leq 1 - \frac{\lambda_i}{\theta} \leq 1 - \frac{\lambda_i}{\mu_j},$$
        because $\lambda_{k+1} \leq \theta \leq \mu_j$. Therefore, for $i = k+1, \ldots, n$, $\lvert \hat{p}(\lambda_i) \rvert \leq \lvert p(\lambda_i) \rvert$.
\end{proofEE}

Now, we can present a result that enables comparing the error in energy norm between the preconditioned and the unpreconditioned iterates.
\begin{theoremE}
\label{theorem : error_UAU_error_A}
Let $\theta \in [\lambda_{k+1}, \lambda_k]$. Then, $
\| x^\ast - \hat{x}_{\ell}(F_{\theta}) \|_{A} \leq \| x^\ast - x_{\ell} \|_{A}$.

\end{theoremE}
\begin{proofEE}
    From~\eqref{eq:ErrorRitz},
\begin{equation}
  \left\| x^\ast - x_{\ell}\right\|_{A}^2 = \min_{p \in \mathbb{P}_\ell(0)} \lVert p \left(A \right) \left( x^\ast - x_0\right) \rVert_{A}^2 = \sum_{i= 1}^{n}\frac{\eta_i^2}{\lambda_i} p_\ell(\lambda_i)^2.\label{eq:error_A}  
\end{equation}
From \Cref{lemma:Poly_comp},
there exists a polynomial $\hat{p}$ of degree $\ell$ with  $\hat{p}(0) = 1$ such that
\begin{align*}
\lvert \hat{p}(\theta) \rvert & \leq \lvert p_{\ell}(\lambda_i) \rvert, \quad i \in \{1, \ldots, k\}\\
\lvert \hat{p}(\lambda_i) \rvert & \leq \lvert p_\ell(\lambda_i) \rvert, \quad i \in \{k+1, \ldots, n\}.
\end{align*}
Applying these inequalities to~\eqref{eq:error_A} yields 
\begin{align*}
   \hspace{1cm} \left\| x^\ast- x_{\ell}\right\|_{A}^2  = \sum_{i = 1}^{n} \frac{\eta_i^2}{\lambda_i} p_{\ell}(\lambda_i)^2 &  \geq \sum_{i =1}^{k} \frac{\eta_i^2}{\lambda_i} \hat{p}(\theta)^2 + \sum_{i =k+1}^{n} \frac{\eta_i^2}{\lambda_i} \hat{p}(\lambda_i)^2\\
          & \hspace{-2cm} \geq \min_{p\in \mathbb{P}_{\ell}(0)} 
          \sum_{i=1}^{k} \frac{\eta_i^2}{\lambda_i}p(\theta)^2 + \sum_{i=k+1}^{n} \frac{\eta_i^2}{\lambda_i}p(\lambda_i)^2 = \left\| x^\ast - \hat{x}_{\ell}(F_{\theta})\right\|_{A}^2.
\end{align*}
\end{proofEE}

\Cref{theorem : error_UAU_error_A} offers a range of choices for $\theta$. Let us remind that to construct $F_{\theta}$, we are given $k$ eigenpairs. As a result, one practical choice is $\theta = \lambda_k$.


In the next section, we present and analyze the  choice of setting $\theta = \lambda_n$.

\subsection{$\theta$ as the smallest eigenvalue}\label{subsec:practical_choice}

For matrices of the form $A = \rho I_n + X$ (with $X \in \mathbb{R}^{n \times n}$ symmetric positive semidefinite and rank deficient), it is common to choose $\theta = \lambda_n$ \citep{LMP}. 
Such a structure of $A$ arises naturally in regularized least-squares problems, where $\lambda_n = 1$ \citep{TshiGratWeavSart08}.
Motivated by this strategy of choosing $\theta$ as the smallest eigenvalue, we analyze in the next theorem how the error in the energy norm along the iterates of $\mathbf{PCG}(F_{\lambda_n})$ can be compared to that of deflated CG.

\begin{theoremE}\label{theorem:CGbehavior}
Let $\varepsilon \geq 0$ be a fixed tolerance and assume that $s_n^\top r_0 = \eta_n \neq 0$. There exists an iterate $\ell_0$ of
deflated CG such that  
    $|\lambda_n - v^{(\ell_0)}_{\ell_0}| \leq \varepsilon$,
    where $v^{(\ell_0)}_{\ell_0}$ is the smallest root of the polynomial $\{p_{\ell_0}^{\text{D}}\} = \arg\min\limits_{p\in \mathbb{P}_{\ell_0}(0)} 
    \sum_{i\in \bar\pi_k^a} \frac{\eta_i^2}{\lambda_i}{p}(\lambda_i)^2.$
    Let $x_\ell^{\text{D}}$ be generated as~\eqref{eq:x_l^D}. Then, the following bounds hold:  
    \begin{equation} \label{eq:SmallEigen}
         \begin{cases}
            \|x^\ast - \hat{x}_\ell(F_{\lambda_n})\|_{A}^2 \leq \|x^\ast - x^{D}_\ell\|_{A}^2 + \sum\limits_{i\in\pi_k^a} \frac{\eta_i^2}{\lambda_i}, & \text{if } \ell < \ell_0, \\[10pt]
            \|x^\ast - \hat{x}_\ell( F_{\lambda_n})\|_{A}^2 \leq \|x^\ast - x^{D}_\ell\|_{A}^2 + \frac{\varepsilon^2}{\lambda_n^2}  \sum\limits_{i\in\pi_k^a} \frac{\eta_i^2}{\lambda_i}, & \text{if } \ell \geq \ell_0.
        \end{cases}
    \end{equation}
\end{theoremE}
\begin{proofEE}
Let us first show the existence of $\ell_0$. Let $(v^{(\ell)}_{j})_{1\leq j\leq \ell}$ denote the roots of $p_\ell^{D}$ given in decreasing order. 
Since $\eta_n \neq 0$, the smallest eigenvalue \(\lambda_n \) will be reached in at most $n-k$ iterations of deflated CG. Therefore, the set $\{\ell \in \{1, \ldots, n-k\} \mid | \lambda_n - v^{(\ell)}_{\ell}| \leq \varepsilon\}$
is non-empty. Let us define $\ell_0 = \min \{\ell \in \{1, \ldots, n-k\} \mid |\lambda_n - v^{(\ell)}_{\ell}| \leq \varepsilon\}$.
At the $\ell$-th iterate, {$\mathbf{PCG}(F_{\lambda_n})$}  satisfies 
\begin{align}
    \|x^\ast - \hat{x}_\ell(F_{\lambda_n})\|_{A}^2 \nonumber &= \min \limits_{p \in \mathbb{P}_{\ell}(0)}\sum_{i\in\pi_k^a} \frac{\eta_i^2}{\lambda_i} p(\lambda_n)^2 + \sum_{i\in\bar\pi_k^a} \frac{\eta_i^2}{\lambda_i} p(\lambda_i)^2\\
    \nonumber & \leq \sum_{i\in\pi_k^a} \frac{\eta_i^2}{\lambda_i} p_\ell^{D}(\lambda_n)^2 + \sum_{i\in\bar\pi_k^a} \frac{\eta_i^2}{\lambda_i} p_\ell^{D}(\lambda_i)^2\\
    &= \sum_{i\in\pi_k^a} \frac{\eta_i^2}{\lambda_i} p_\ell^{D}(\lambda_n)^2 + \|x^\ast - x^{D}_\ell\|_{A}^2.   
    \label{enorm_theta_seigen}
\end{align}  
For \( \ell < \ell_0 \), since
$
p_\ell^{D}(\lambda_n)^2 = \prod\limits_{j=1}^{\ell} \left(1 - \frac{\lambda_n}{v^{(\ell)}_{j}}\right)^2  \leq 1,
$
we obtain from~\eqref{enorm_theta_seigen}
\[
\|x^\ast - \hat{x}_\ell(F_{\lambda_n})\|_{A}^2 \leq \|x^\ast - x^{D}_\ell\|_{A}^2 + \sum\limits_{i\in\pi_k^a} \frac{\eta_i^2}{\lambda_i}.
\] 
For \( \ell \geq \ell_0 \), by using the interlacing property,
\[
p_\ell^{D}(\lambda_n)^2 
\leq \left(1 - \frac{\lambda_n}{v^{(\ell)}_{\ell}}\right)^2 \leq \left(1 - \frac{\lambda_n}{v^{(\ell_0)}_{\ell_0}}\right)^2
\leq \frac{\varepsilon^2}{\lambda_n^2}.
\]  
As a result, from~\eqref{enorm_theta_seigen},
$\|x^\ast - \hat{x}_\ell(F_{\lambda_n})\|_{A}^2 \leq \|x^\ast - x^{D}_\ell\|_{A}^2 + \frac{\varepsilon^2}{\lambda_n^2} \sum\limits_{i\in\pi_k^a} \frac{\eta_i^2}{\lambda_i}$.
\end{proofEE}
  
\cref{theorem:CGbehavior} shows two-phase convergence behavior. It states that  if $\ell< \ell_0$, the terms  
$\sum_{i \in \pi_k^a} \frac{\eta_i^2}{\lambda_i}$  
may slow down convergence and even lead to worse results compared to the unpreconditioned case.  
However, if $\ell \geq \ell_0$, the behavior is expected to resemble that of deflated CG.

\section{Numerical experiments}
\label{sec:NumExp}

\subsection{Approximating largest and smallest eigenvalues}

In practice, the largest eigenvalues can be approximated when solving sequences of linear least-squares problems arising in nonlinear least-squares problems~\citep{TshiGratWeavSart08}.
The Lanczos process underlying CG can be exploited on each
linear problem to extract approximate spectral information~\citep{YoussSaad}.  This spectral information is used to design a preconditioner for the
subsequent linear least-squares problem~\citep{TshiGratWeavSart08}. Another approach to approximate the largest eigenvalues is randomized eigenvalue decomposition~\citep{Halko}, which has been shown to be efficient for constructing preconditioner to solve SPD linear
systems~\citep{frangella2023randomized}.

For the smallest eigenvalues, in the context of solving linear systems with
multiple right-hand sides, \citet{Stathopoulos} proposed a CG-based approach
to approximate the smallest eigenvalues and their corresponding eigenvectors
relying on Rayleigh--Ritz projections  and to reuse them as a deflation subspace for later linear systems. There is another way to approximate the smallest eigenvalues and their corresponding eigenvectors by using  harmonic projection techniques~\citep{DeflatedCG, parks2006recycling}. A
comparative study between Rayleigh--Ritz and harmonic projections in the context
of approximating the smallest eigenpairs was presented
in~\citep{venkovic2020comparative}.

\subsection{Computational Complexity}
The computational cost of calculating the eigenpairs depends on the strategy employed and on which eigenpairs are targeted. For instance, approximating eigenspectrum from CG coefficients at iteration $k$ requires $\mathcal{O}(nk^2)$ flops, which accounts for the eigendecomposition of a $k \times k$ tridiagonal matrix as well as the matrix-vector products with the $n$-dimensional Lanczos vectors (normalized residual vectors produced by CG) \citep{DataAssimilaton}. When randomized algorithms are used to approximate the eigenpairs, the overall cost is dominated by at least one matrix–vector product with $A$ \citep{Halko}. A key advantage of randomized methods is that they are parallelizable and can provide spectral information in advance for PCG.

Applying $F_{\theta}$ requires storing the $k$ selected eigenvectors $\{s_i\}_{i\in\pi_k^a}$, the associated eigenvalues $\{\lambda_i\}_{i\in\pi_k^a}$, and the scalar parameter $\theta$. This results in a memory cost of $\mathcal{O}(kn)$.
The total computational cost of performing matrix-vector products with $F_{\theta}$ is $\mathcal{O}(kn)$ flops~\citep{LMP}.
    There is an extra
cost associated with constructing $\theta_1$, but it is dominated by a single
multiplication with the matrix $A$.

\subsection{Experimental setup}
\label{subsection:Experimentalsetup}
We restrict our attention to \( Ax = b \), where \( A \in \mathbb{R}^{n \times n} \) is a diagonal matrix with \( n = 10^6 \). In this setting, the eigenvalues used to build the preconditioner are exact, and the preconditioner \( F_\theta \) is therefore diagonal. We consider the eigenvalue distribution~\citep{StrakosGreenbaum}:

\begin{equation}
    \label{eq:strakos_eigenvalues}
    \lambda_{i} = \lambda_n + \left( \frac{n - i}{n - 1} \right) (\lambda_1 - \lambda_n) \rho^{i-1}, \quad \text{for} \quad i = 1, \ldots, n,
\end{equation}
with $\lambda_1 = 10^6, \quad \lambda_n = 1$ and $\rho = 0.75$. The preconditioner \( F_{\theta} \) is constructed using \( k\in\{30,40,50\}\) largest eigenvalues. We denote $\theta_{\mathrm{m}} = \bigl(\lambda_{k}+\lambda_{n}\bigr)/2$ and $\theta_{\mathrm{r}} =\lambda_k$. 

The choices for $\theta_{\mathrm{m}}$ and $\theta_{\mathrm{r}}$ are motivated from \Cref{SubSec:median}. In addition, further theoretical results are provided for $\theta_{\mathrm{r}}$ in \Cref{SubSec:Case1}.

We use $b = [1, \ldots, 1]^\top/\sqrt{n}$, and $x_0 = 0$. We compare the performance of the methods of \Cref{tab:methods} in terms of the relative error in energy norm at each iteration. 

\begin{table}[ht]
\centering
\begin{tabular}{ll}
\hline
Method & Description \\ \hline
CG & \Cref{Algo:CG} applied to $Ax = b$ \\
$\mathbf{PCG}(F_{\theta_\mathrm{r}})$
& \Cref{Algo:PCG} applied to $Ax = b$ using $F = F_{\theta_\mathrm{r}}$ \\ 
$\mathbf{PCG}(F_{\theta_\mathrm{1}})$
 &
\Cref{Algo:PCG} applied to $A x = b$ using $F = F_{\theta_\mathrm{1}}$\\
$\mathbf{PCG}(F_{\theta_\mathrm{m}})$
 & \Cref{Algo:PCG} applied to $A x = b$ using $F = F_{\theta_\mathrm{m}}$ \\
\textbf{DefCG} &  \Cref{Algo:CG} applied to the projected system~\eqref{eq:projected_sys}\\ \hline
\end{tabular}
\caption{Description of methods used in the numerical experiments.}
\label{tab:methods}
\end{table}

In all plots below, the eigenvalues are indexed by $u = (i - \tfrac{1}{2})/N$ for $i=1,\dots,N$,
so that $u \in (0,1)$ represents the relative position of the $i$-th
eigenvalue in the ordered spectrum. The abscissa is displayed on a
logit scale, $\mathrm{logit}(u) = \log \!\left(u/(1-u)\right)$,
which expands both ends of the interval $(0,1)$. This transformation
enhances the visibility of extreme eigenvalues (corresponding to very small or very large indices) by spreading them apart, while keeping the bulk of the spectrum compressed in the center. 

In the supplementary material \citep{supplementary2026}, we present experiments for other eigenvalue distributions, allowing us to recover the different cases described in \Cref{subsec:pi_k}. 

\subsection{Numerical results}
\label{Sec:NumericalResultsC1}
\Cref{fig:larger} shows eigenvalue distributions of the preconditioned matrix $FA$ for different choices for $k$ and $F$. 
It also shows the convergence of {$\mathbf{PCG}(F)$} in terms of the error in the relative energy norm. Note that the CG does not depend on $k$; for this reason, and for clarity of presentation, the corresponding CG results are shown only in the first column of \Cref{fig:larger} (i.e., $k=30$).


\begin{figure}[htb]
\centering
\includegraphics[width=1.\linewidth]{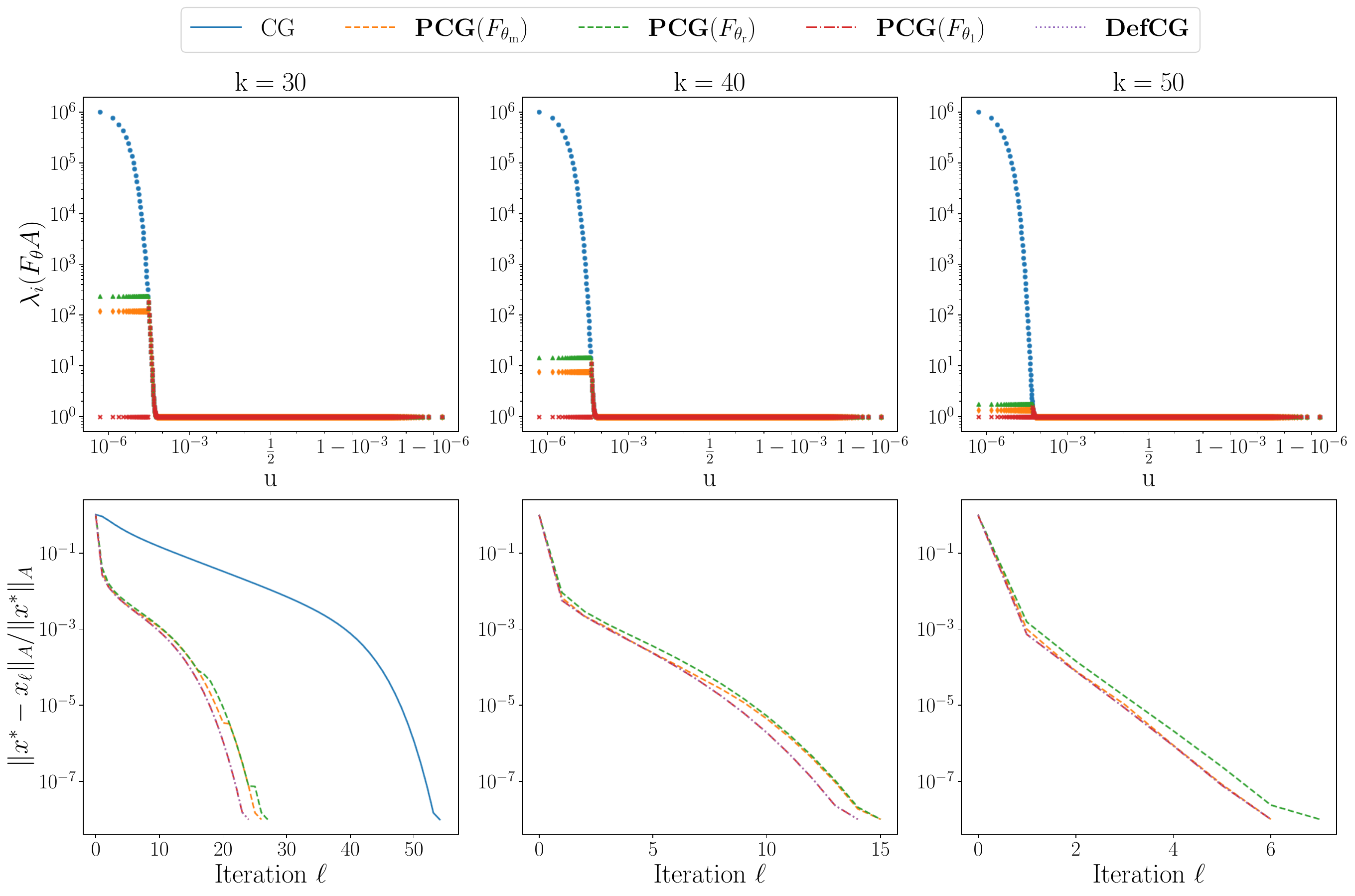}
    \caption{Top: Eigenvalue distributions of $F_{\theta}A$ for $k = 30$, $40$, and $50$ (from left to right).
Bottom: Relative errors in energy norm versus iteration $\ell$, including \textbf{DefCG}.}
\label{fig:larger}
\end{figure}

As shown in \Cref{fig:larger}, all methods require fewer iterations than the unpreconditioned case to achieve an approximation of the solution for which $\|x^\ast - x_\ell\|_A \approx 10^{-8}\|x^\ast\|_A$.  $\mathbf{PCG}(F_{\theta_\mathrm{r}})$  achieves better convergence than CG at all iterates, as guaranteed by~\Cref{theorem : error_UAU_error_A}.
 $\mathbf{PCG}(F_{\theta_\mathrm{1}})$ achieves better convergence than CG and also outperforms both $\mathbf{PCG}(F_{\theta_\mathrm{r}})$ and $\mathbf{PCG}(F_{\theta_\mathrm{m}})$. Its convergence remains close to that of \textbf{DefCG}.
$\mathbf{PCG}(F_{\theta_\mathrm{m}})$ produces iterates closer to those of \textbf{DefCG}, as this behaviour is explicitly motivated for the choice of $\theta_{m}$ as explained in \Cref{SubSec:median}. In addition, $\mathbf{PCG}(F_{\theta_\mathrm{m}})$  outperforms slightly $\mathbf{PCG}(F_{\theta_\mathrm{r}})$, which can be explained by $\frac{\alpha(\theta_{\mathrm{m}})}{\theta_{\mathrm{m}}} \leq \frac{\alpha(\theta_{\mathrm{r}})}{\theta_{\mathrm{r}}}$ (See \Cref{theorem:midrange}).

\subsection{Discussion}

Both the theoretical analysis and the numerical experiments suggest that deflated CG gives the best result. However, in practice, since only approximate eigenpairs are typically available, deflated CG can become computationally demanding for large-scale problems. Constructing the deflation projector \citep{DeflatedCG2} requires storing additional vectors and involves extra multiplications with $A$, as well as additional vector operations. This results in non-negligible overhead, both in memory usage and in the number of matrix-vector products.

One motivation for the scaled spectral preconditioner is to mimic the favorable convergence behavior of deflated CG while significantly reducing this overhead. In follow-up work, we will provide theoretical and numerical result showing that the scaled spectral preconditioner remains robust when approximate eigenpairs are used and can achieve an approximation close to deflated CG at low cost.

\subsection{Numerical experiments with the choice of $\theta = \lambda_n$}
To show the behavior in \Cref{subsec:practical_choice}, we consider the same distribution as~\eqref{eq:strakos_eigenvalues} with $n=100, \rho =0.75, \lambda_1 = 10^4$ and $\lambda_n = 1$. The preconditioner uses the $k=10$ largest eigenvalues. 

We consider two different choices for the right-hand side vector \( b \), based on the distribution of \( \zeta_i := \frac{\eta_i^2}{\lambda_i} \), defined as follows: (i) \textbf{Fast Decay of \( \zeta_i \):}  
    a fast decay in the values of \( \zeta_i \) is given by 
   \( \zeta_i = \zeta_n + ( \frac{i - 1}{n - 1} ) (\zeta_1 - \zeta_n) 0.9^{n-i} \).
    We set \( \zeta_N = 1 \) and \( \zeta_1 = 10^3 \). The right-hand side vector is then defined as $
    b = [\sqrt{\zeta_1 \lambda_1}, \sqrt{\zeta_2 \lambda_2}, \ldots, \sqrt{\zeta_n \lambda_n}]$. (ii) \textbf{Fast Growth of \( \zeta_i \):}  
    we reverse the order of \( \zeta_i \) defined in the previous case to obtain a fast-growing distribution: $
    b = [\sqrt{\zeta_{n} \lambda_1}, \sqrt{\zeta_{n-1} \lambda_2}, \ldots, \sqrt{\zeta_{1} \lambda_n}]$.

\begin{figure}[ht]
\centering
\includegraphics[width=1.\linewidth]{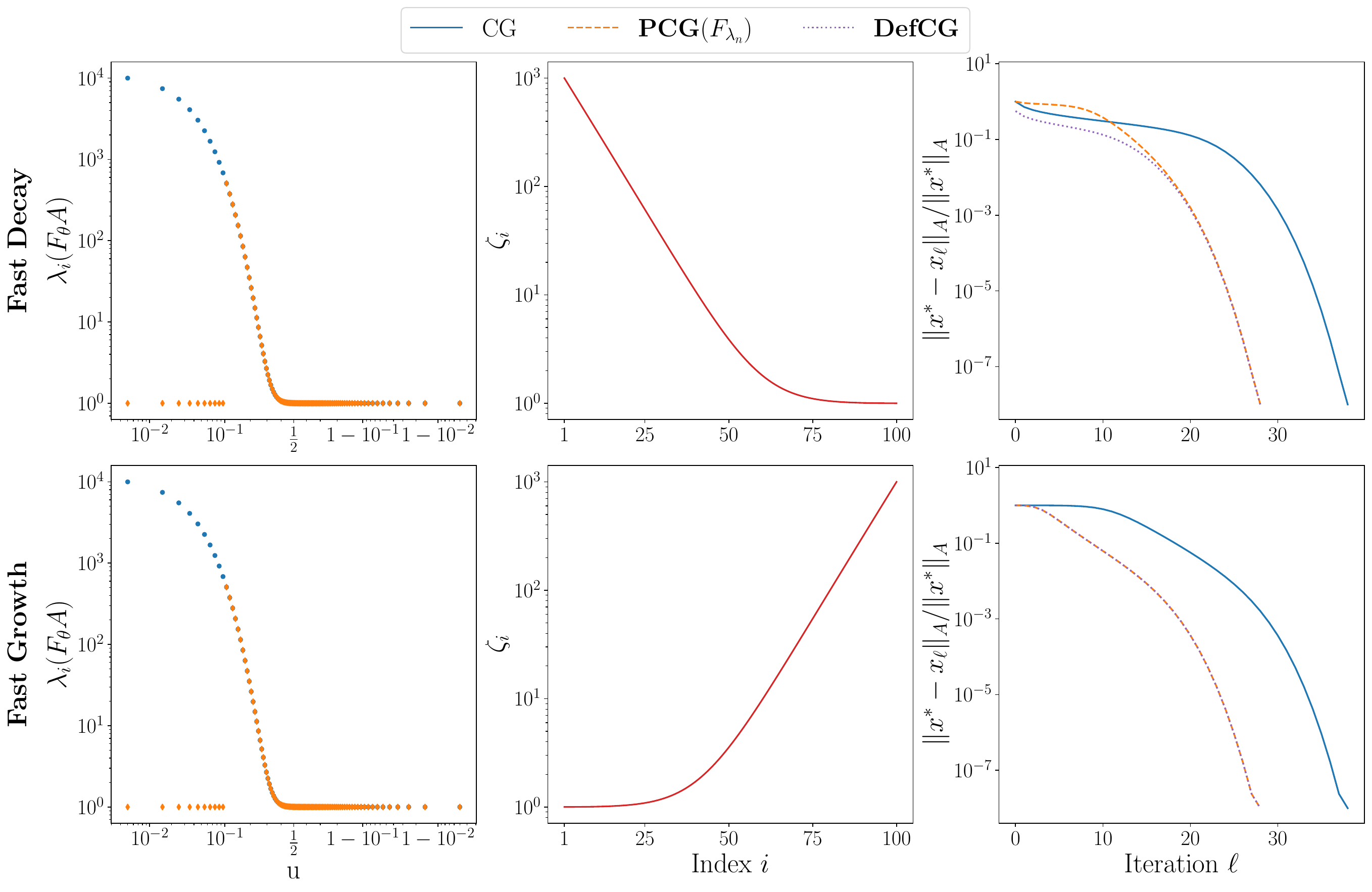}
    \caption{
    Eigenvalues of the preconditioned matrix, the values of $\zeta_i$, and the energy norm of the relative error along  iterations for \textbf{fast decay} and \textbf{fast growth} of $\zeta_i$.
    }
    \label{fig:rhs_dec}
\end{figure}


As shown in \Cref{fig:rhs_dec}, $\mathbf{PCG}(F_{\lambda_{n}})$ reduces the number of iterations to converge, but it is not necessarily better than CG, especially for $\ell < 15$. We observe that the convergence improves for $\ell \ge 15$ and becomes closer to \textbf{DefCG} for $\ell \ge 20$, this can be interpreted as a consequence of the convergence of the smallest Ritz value toward $\lambda_n$ in \textbf{DefCG}, in accordance with the results of \Cref{theorem:CGbehavior} (see the supplementary material~\cite[Fig. SM3]{supplementary2026}). Similarly, the slow convergence of $\mathbf{PCG}(F_{\lambda_{n}})$ in the early iterations can be explained by \Cref{theorem:CGbehavior} where 
$\sum_{i=1}^{k} \zeta_i $ is large enough to slow down the convergence. For the case with a fast growth distribution of $\zeta_i$, $\mathbf{PCG}(F_{\lambda_{n}})$  exhibits faster convergence in the early iterations, similar to \textbf{DefCG}, because the terms $\sum_{i=1}^{k}\zeta_i $ in
\eqref{eq:SmallEigen} 
remain very small.

\section{Conclusion and perspectives}
\label{sec:conclusion} 
We have introduced and analyzed a class of scaled spectral preconditioners for the conjugate gradient method applied to large symmetric positive-definite linear systems with extreme eigenvalues, {particularly when the CG number of iterations is limited.} Starting from an optimization viewpoint, we derived a preconditioner that minimizes the error after a prescribed number of iterations. A key outcome of our analysis is that the preconditioner should be built from the extreme eigenvalues, and that the position of the resulting eigenvalue cluster can be chosen according to several principled criteria. The numerical experiments on matrices with extreme eigenvalues confirm that the scaled spectral preconditioners can significantly accelerate the rate of convergence of PCG.

In follow-up work, we will provide a detailed theoretical analysis of the preconditioner when it is used for a sequence of SPD matrices, or equivalently, when it is constructed from approximate eigenvalues.
\newpage

\appendix

\section{Proofs}
\label{appendix:Proofs}
\printProofs
\section{Properties of the projected system~\eqref{eq:projected_sys}}
\label{appendix:projected_sys}

\begin{proposition}\label{prop:PA=AP}
The matrix $PA$ defined in~\eqref{eq:projected_sys} is symmetric positive semi-definite. In addition,
\begin{itemize}
    \item $PA = AP = PAP.$
    \item The system $PA z = Pb$ is consistent.
\end{itemize}   
\end{proposition}
\begin{proof}
    See \citep{coulaud2013deflation}. Note that in \citep{coulaud2013deflation}, $P$ is defined for a general deflation subspace $W$. 
\end{proof}

\begin{proposition}\label{prop:reduced_system}
   The iterate $z_\ell$ generated by CG when solving~\eqref{eq:projected_sys} starting with $z_0 = x_0$ satisfy $$z_\ell = \sum_{i\in\pi_k^a} \, s_i s_i^\top x_0 +  S_{\bar\pi_k^a} \hat{y}_\ell,$$ where $S_{\bar\pi_k^a} = [s_i]_{i\in\bar\pi_k^a}$. Here, $\hat{y}_\ell$ generated with CG when solving $\Lambda_{\bar\pi_k^a} \hat{y} = S_{\bar\pi_k^a}^\top b$ starting with $\hat{y}_0 = S_{\bar\pi_k^a}^\top x_0$ where $\Lambda_{\bar\pi_k^a} = \diag((\lambda_i)_{i\in\bar\pi_k^a})$.
\end{proposition}
\begin{proof}
    From the expression of $PA$, it is easy to see that $PA= \sum_{i\in\bar\pi_k^a} \, \lambda_i s_i s_i^\top$.
In addition $ P \;b = (I - \sum_{i\in\pi_k^a} \, s_i s_i^\top)b=  \sum_{i\in\bar\pi_k^a} \, s_i s_i^\top b = S_{\bar\pi_k^a}^\top S_{\bar\pi_k^a} b$. The proof follows directly from \citep{hayami2018convergence} by decomposing CG onto the range and null space of $PA$.
\end{proof}

\newpage
\bibliographystyle{plainnat}  
\bibliography{references}

\end{document}